\documentclass[letterpaper,10pt,conference]{ieeeconf}  % Comment this line out
\IEEEoverridecommandlockouts
\let\labelindent\relax
\usepackage{enumitem,graphicx,xspace}
\usepackage{subfig}

\usepackage{amsmath,amssymb, amsfonts,cases,mathtools,bbm,bm,bbold} 
\newcommand{\pde}[2]{ \frac{\partial #1}{\partial #2} }
\newcommand{\ppde}[2]{\frac{\partial^2 #1}{\partial #2^2}}

\newcommand{\R}{\mathbb{R}}
\newcommand{\N}{\mathbb{N}}
\newcommand{\1}{\mathbb{1}}

\newcommand{\diag}[1]{\mathrm{diag}\left(#1\right)}
\newcommand{\sign}[1]{\mathrm{sign}\left(#1\right)}
\newcommand{\card}[1]{\mathrm{card}(#1)}

\newcommand{\abs}[1]{ \left\lvert #1 \right\rvert }
\newcommand{\norm}[1]{ \| #1 \|}
\newcommand{\spectrum}[1]{\mathrm{sp}(#1)}
\newcommand{\G}{\mathcal{G}}
\newcommand{\edgeSet}{\mathcal{E}}
\newcommand{\nodeSet}{\mathcal{V}}

\newcommand{\Gh}{\G^{(h)}}
\newcommand{\edgeSeth}{\edgeSet^{(h)}}

\newcommand{\vspan}[1]{\mathrm{span}(#1)}

\newtheorem{theorem}{Theorem}
\newtheorem{lemma}{Lemma}
\newtheorem{assumption}{Assumption}
\newtheorem{remark}{Remark}

\usepackage[colorlinks=true,citecolor=blue,linkcolor=blue,urlcolor=blue,]{hyperref}

\usepackage{comment}

\newcommand{\A}{{A^{(2)}}}
\newcommand{\Bi}[1]{{A^{(3)}_{#1}}}
\newcommand{\Bd}{{A^{(3)}}}

\newcommand{\xeq}{x^\ast}
\allowdisplaybreaks

\title{\LARGE\bf Collective decision-making dynamics in hypernetworks}

\author{
Angela Fontan$^{1 \ast}$ and Silun Zhang$^2$
\thanks{
$^1$A. Fontan is with the Division of Decision and Control Systems, KTH Royal Institute of Technology. 
E-mail: angfon@kth.se. $^2$S. Zhang is with the Department of Mathematics, KTH Royal Institute of Technology. 
E-mail: silunz@kth.se. 
The authors are also affiliated with Digital Futures and WASP. $^\ast$Corresponding author.
}
\thanks{This work was supported by the Knut and Alice Wallenberg Foundation Wallenberg AI, Autonomous Systems and Software Program (WASP), and by the Digital Futures Research Pairs Project.}
}

\begin{document}

\maketitle
\begin{abstract}
This work describes a collective decision-making dynamical process in a multiagent system under the assumption of cooperative higher-order interactions within the community, modeled as a hypernetwork.
The nonlinear interconnected system is characterized by saturated nonlinearities that describe how agents transmit their opinion state to their neighbors in the hypernetwork, and by a bifurcation parameter representing the community's social effort.
We show that the presence of higher-order interactions leads to the unfolding of a pitchfork bifurcation, introducing an interval for the social effort parameter in which the system exhibits bistability.
With equilibrium points representing collective decisions, this implies that, depending on the initial conditions, the community will either remain in a deadlock state (with the origin as the equilibrium point) or reach a nontrivial decision.
A numerical example is given to illustrate the results.
\end{abstract}

\section{Introduction}
In recent years, there has been a growing interest in understanding the collective behavior of networked systems with higher-order interactions (see, e.g., \cite{Battiston2020networks-beyond} for a review). The motivation stems from practical needs in real-world applications where the phenomena observed in complex systems cannot be adequately captured by considering only pairwise interactions between agents. Instead, capturing such phenomena requires incorporating higher-order interactions, which are often represented by using hypergraphs or simplicial complexes. Such a need is evident in numerous applications \cite{Gambuzza2021stability}, ranging from neuron dynamics \cite{Battiston2020networks-beyond} to protein interaction networks \cite{Estrada2018centralities}, and from ecological systems \cite{Mickalide2019higher} to social systems \cite{Battiston2020networks-beyond,Sahasrabuddhe2021nonlinear,Sarker2023homophily,Hickok2022boundedconfidence,Neuhauser2022Consensus}. Most research in this area has focused on the emergence of consensus \cite{Neuhauser2022Consensus,Neuhauser2020multibody} and synchronization \cite{Battiston2020networks-beyond,Gambuzza2021stability,Gallo2022synchronization}, where interconnected agents converge over time to a common equilibrium point or solution, respectively.

In this work, we want to understand how higher-order interactions influence collective behavior in opinion formation under social influence, using hypernetwork dynamics characterized by nonlinear (sigmoidal) functions. 
Extensive research on opinion dynamics in social networks has examined various sociological phenomena, including consensus formation, polarization, persistent disagreement, and clustering, see \cite{Proskurnikov2017Tutorial,Proskurnikov2018Tutorial,Shi2019Dynamics} for reviews.
There have been preliminary efforts to extend well-known opinion dynamics models, like the French-DeGroot model for consensus \cite{DeGroot1974Reaching}, to incorporate higher-order interactions (see, e.g., \cite{Neuhauser2020multibody,Hickok2022boundedconfidence,Zhang2023higher}). In addition, it has been proven that nonlinearities are necessary to capture higher-order effects, as opinion dynamics with linear interactions can be reduced to a pairwise representation \cite{Neuhauser2020multibody}. However, the investigation of these nonlinear dynamics on hypernetworks remains limited. 
We here focus on a particular class of decision-making systems, where nonlinear mechanisms--such as saturation and sigmoidal functions--are used to model how individuals express and transmit their opinions within a network.
These models were previously introduced for lower-order networks in \cite{Abara2018Spectral,Fontan2018Multiequilibria,GrayAl2018}, and are inspired by Hopfield neural networks \cite{Hopfield1984}.
They are used, for instance, to model opinion formation and collective decision-making in both cooperative \cite{Fontan2018Multiequilibria} and antagonistic \cite{Fontan2021Role,Bizyaeva2021Nonlinear,Bizyaeva2025multitopic} communities, with applications ranging from animal group behavior \cite{GrayAl2018} to political decision-making \cite{Fontan2021Signed}. Moreover, modeling societal-scale networks with such saturated nonlinearities typically demands additional effort, e.g., system identification \cite{zhang2025online} and reduced-order modeling \cite{zhang2020modeling}.
 
In this work, we extend these dynamical models to hypernetworks of order $h$ and investigate their qualitative behavior for $h=2$. These hypernetworks include not only pairwise interactions between agents (referred to as 1-order interactions or simply 1-interactions) but also triangle-like interactions involving three agents (referred to as 2-interactions). Similar models have been analyzed, for instance, in \cite{He2020GlobalImpulses} (and referenced therein), where the authors study the stability of higher-order Hopfield neural networks with state-dependent impulses.
Unlike these works, but in line with \cite{Fontan2018Multiequilibria,Fontan2021Role}, our model is characterized by a Laplacian-like structure at the origin and considers monotone saturating nonlinearities to describe the interaction terms.
A bifurcation parameter $\pi$, representing the social effort of the community or strength of interactions, modulates these interactions. 

In the absence of 2-interactions, the qualitative collective behavior is described by a pitchfork bifurcation that occurs at a threshold value $\pi_1$ of the bifurcation parameter. Below this value, the system remains in a deadlock state, where the origin is the unique equilibrium point. This deadlock is resolved only when the social effort exceeds the critical threshold $\pi_1$, at which the dynamical system admits two nontrivial equilibrium points, representing possible collective decisions for the community.

We show that, when present, 2-interactions introduce asymmetry and lead to an unfolding of the pitchfork bifurcation, which indicates a transition from a saddle-node bifurcation to the pitchfork bifurcation. In particular, as $\pi$ increases, the system transitions from having the origin as the unique equilibrium point to undergoing a saddle-node bifurcation at a critical threshold value $\pi_1^\ast$ of $\pi$, where two equilibria emerge: one unstable and one locally asymptotically stable. At a second critical threshold value $\pi_1$ of $\pi$, the unstable point of the saddle-node bifurcation collides with the origin, and both equilibria change stability. As a result, in the interval $(\pi_1^\ast,\pi_1)$, the nonlinear dynamical system exhibits bistability.

The remainder of the paper is organized as follows. In Section~\ref{sec:notation-prelim}, we introduce the notation and mathematical preliminaries needed for our analysis. Section~\ref{sec:dynamics} presents the nonlinear dynamical model used for describing the decision-making process over hypernetworks, and Section~\ref{sec:main-results} is dedicated to a detailed analysis of the existence and stability of the equilibria in such a dynamical model. Section~\ref{sec:example} provides a numerical example illustrating the bifurcation behaviors in the analysis. Finally, Section~\ref{sec:conclusion} concludes the paper.

\section{Notation and Technical Preliminaries}\label{sec:notation-prelim}
This section presents the notation used in the paper (Section~\ref{sec:notation}), as well as technical preliminaries on matrix theory (Section~\ref{sec:matrix-theory}), graph and hypergraph theory (Section~\ref{sec:nets-theory}).

\subsection{Notation}\label{sec:notation}
$\N$ and $\R$ indicate the set of natural and real numbers, respectively. $\1_n\in \R^n$ is the vector of 1s, $I_n$ is the identity matrix of dimension $n$. For a vector $x$, $\abs{x}$ indicates its element-wise absolute value. 

\subsection{Linear Algebra}\label{sec:matrix-theory}
Given a matrix $A= [a_{ij}] \in \R^{n\times n}$, we use $A\ge 0$ and $A> 0$ to denote that A is elementwise nonnegative and positive, respectively. The spectrum of $A$ is denoted by $\spectrum{A}= \{\lambda_1(A),\dots,\lambda_n(A)\}$, where $\lambda_i(A)$, $i=1,\dots,n$, are the eigenvalues of $A$. 
A matrix $A$ is (diagonally) symmetrizable if $DA$ is symmetric for some diagonal matrix $D$ with positive diagonal entries. If $A$ is symmetric (or symmetrizable), then it has real eigenvalues, which we assume to be arranged in a nondecreasing order, i.e., $\lambda_1(A)\le \lambda_2(A) \le \dots \le \lambda_n(A)$.
A matrix $A$ is irreducible if there does not exist a permutation matrix $P$ s.t. $P^TAP$ is block triangular.

\subsection{Graphs and hypernetworks}\label{sec:nets-theory}
This section briefly introduces the definition of a hypergraph and related notions. For details, refer to \cite{Battiston2020networks-beyond,Courtney2016generalized-nets}. 

We begin with the formal definition of graphs. Let $\G = (\nodeSet, \edgeSet)$ be a graph with vertex set $\nodeSet$ (such that
$\card{\nodeSet}= n$) and edge set $\edgeSet \subseteq \nodeSet\times \nodeSet$. A graph is undirected if $(i,j)\in \edgeSet$ implies $(j,i)\in \edgeSet$, and connected if there exists a path, i.e., a sequence of edges, from each node $i\in \nodeSet$ to any other node $j\in \nodeSet$.
The adjacency matrix of $\G$ is $A=[a_{ij}]\in \R^{n\times n}$ where $a_{ij}=0$ if and only if $(j,i)\in \edgeSet$. We consider undirected and connected graphs without self-loops, which means that $A=A^T$ is irreducible and $a_{ii}=0$ for all $i=1,\dots,n$. 
The diagonal matrix $\Delta = \diag{\delta_1,\dots,\delta_n}:=\diag{A\1_n}$ is the degree matrix of $\G$. Since $\G$ is connected and undirected, $\delta_i>0$ for all $i=1,\dots,n$.

The graph $\G$ captures pairwise interactions between nodes, referred to as 1-order interactions or simply 1-interactions. A $0$-interaction represents a self-loop, which is not considered in this paper. 
One way to describe higher-order interactions between nodes, referred to as $h$-interactions for $h\ge 2$, is to use hypergraphs. 

Let $\Gh = (\nodeSet, \edgeSeth)$ be a hypergraph with vertex set $\nodeSet$ (such that
$\card{\nodeSet}= n$) and hyperedge set $\edgeSeth$. Each hyperedge can be expressed as a $k$-interaction $(j_1,j_2,\dots,j_{k+1})$ ($k\in \N, \,k\le h$), between the nodes $j_1,j_2,\dots,j_{k+1} \in \nodeSet$. Note that a graph $\G$ with only 1-interactions can be denoted as $\G^{(1)}$. 
The following representation of the adjacency matrix for a hypergraph $\Gh$ generalizes the notion introduced above for the graph $\G$, following the notation adopted in \cite[p.~10]{Battiston2020networks-beyond}. For each $d = 2,\dots,h+1$, define the $\underbrace{n\times n \cdots \times n}_{d}$ adjacency tensor $A^{(d)}$ such that $a^{(d)}_{j_1 \dots \,j_{d}} \ne 0 $ if and only if the $d$ nodes $j_1, \dots \,j_{d}$ are participating in a $(d-1)$-interaction. 

A more compact representation of the adjacency tensor $A^{(d)}$ capturing $(d-1)$-interactions is: 
$a^{(d)}_{\alpha_d}\ne 0$ iff $\alpha_d\in \edgeSeth$. 
The degree of each node $i\in \nodeSet$ (adapted from the notions of generalized degrees introduced in \cite{Courtney2016generalized-nets,Battiston2020networks-beyond}) is
\begin{equation}
    \delta_i:= \sum_{d=2}^{h+1}\sum_{\alpha_d \in \edgeSeth:\, \{i\}\subseteq \alpha_d } a^{(d)}_{\alpha}.
\label{eqn:degree-i}
\end{equation}
Intuitively, \eqref{eqn:degree-i} indicates the (weighted) number of $(d-1)$-interactions ($\alpha$) that are adjacent to $i$, $d=2,\dots,h+1$.
When $h=1$, each $\alpha\in \mathcal V^2$, and \eqref{eqn:degree-i} reduces to $\delta_i = \sum_{j} a_{ij}$, which is equivalent to the definition of the node degree of a simple graph $\G$. 
Particularly, when $h=2$, 
$\delta_i = \sum_{j} a^{(2)}_{ij} + \sum_{j,k} a^{(3)}_{ijk}$, which aligns with similar formulations found in \cite{Courtney2016generalized-nets, Tudisco2021centrality}.

\section{Decision-making in hypernetworked nonlinear systems}\label{sec:dynamics}
To model the decision-making process in a community of $n$ agents represented by a hypernetwork $\Gh$, we consider the following class of nonlinear high-order interconnected systems
\begin{multline}
\dot x_i = - \delta_i x_i + \pi \sum_{d=2}^{h+1}\;\sum_{j_1,\dots,j_{d-1}=1}^n \!\!\!\!\!a^{(d)}_{i j_1\dots j_{d-1}} \prod_{k=1}^{d-1}\psi_{j_k}(x_{j_k}), \\\quad i=1,\dots n,
\label{eqn:model_high-order}
\end{multline}
Here, $x_i$ represents the opinion of agent $i$, the weights $a^{(d)}_{i j_1\dots j_{d-1}}$ are the elements of the adjacency tensors of the hypernetwork $\Gh$, capturing the $(d-1)$-interactions among agents. The term $\delta_i$ represents the (node) degrees of $\Gh$ and reflects the inertia or resistance of agent $i\in \nodeSet$ to changes in opinion. The parameter $\pi$ is a positive scalar representing the social effort or strength of commitment of the community (see e.g. \cite{Fontan2018Multiequilibria}), $\psi_1,\dots,\psi_n$ are nonlinear functions describing how the agents transmit their opinions to their hyperedges neighbors in $\Gh$. 
The parameter $\pi$ assumes the role of bifurcation parameter in the analysis presented in Section~\ref{sec:main-results}.

\begin{remark}
The model~\eqref{eqn:model_high-order} can be viewed as a generalization of nonlinear network dynamics (corresponding to $h=1$), such as those studied in honeybee-inspired cooperative networks~\cite{GrayAl2018,Fontan2018Multiequilibria} and signed networks~\cite{Fontan2021Role}, to include higher-order interactions captured by hypergraphs. For $h=1$, see, e.g., \cite{GrayAl2018,Fontan2018Multiequilibria}, model~\eqref{eqn:model_high-order} reduces to 
\begin{equation}
\dot x_i = - \delta_i x_i + \pi \;\sum_{j=1}^n a_{i j} \psi_{j}(x_{j}),\quad i=1,\dots n.
\label{eqn:model_h1}
\end{equation}
\end{remark}

In this paper, we focus on the case of $h=2$, in which the hypernetworked nonlinear system~\eqref{eqn:model_high-order} becomes:
\begin{multline}
\dot x_i = - \delta_i x_i + \pi 
\sum_{j=1}^n a^{(2)}_{i j} \psi_{j}(x_{j})
+\pi \sum_{j,k=1}^n a^{(3)}_{i j k} \psi_{j}(x_{j})\psi_{k}(x_{k}),\\\quad i=1,\dots n.
\label{eqn:model_i}
\end{multline}
For each agent $i\in \nodeSet$, let $\Bi{i}:=\Bd(i,\,:,\,:)=[a^{(3)}_{i j k}]_{j,k=1}^n\in \R^{n\times n}$, and, with some abuse of notation, rewrite $\Bd$ as $\Bd = \diag{\Bi{1},\dots,\Bi{n}}$. Then, the individual dynamics~\eqref{eqn:model_i} becomes
\begin{align*}
\dot x_i &= - \delta_i x_i + \pi \sum_{j=1}^n \left( a^{(2)}_{i j} \psi_{j}(x_{j})+\psi(x)^T\Bi{i} \psi(x)\right),
\end{align*}
which can be further rewritten in a compact form:
\begin{align}
\dot x 
&= - \Delta x + \pi \Big(\A \psi(x) + (I_n \otimes \psi(x)^T) \Bd (\1_n \otimes \psi(x)) \Big)
\label{eqn:model}
\end{align}
where $x=[x_1\,\cdots\,x_n]^T\in \R^{n}$ is the vector of agents' opinions\footnote{For simplicity, we consider scalar-valued agent opinions; the extension to higher-dimensional agent states can be handled analogously.}, the matrices $\Delta$, $\A$, and $\Bd$ are the degree matrix and the matrices corresponding to the adjacency tensors of $\Gh$, respectively, and $\psi(x)=[\psi(x_1)\,\cdots\,\psi(x_n)]^T\in \R^{n}$. Similarly to the case $h=1$ in \cite{GrayAl2018,Fontan2018Multiequilibria}, we assume that $\A$ is symmetric, irreducible, and with null diagonal.

Next, we introduce the main assumptions on the nonlinearity and the hypergraph structure in model \eqref{eqn:model}. We begin with an assumption that captures a class of nonlinearities commonly encountered in opinion dynamics, i.e., 
``S-shaped'' functions, such as the hyperbolic
tangent and sigmoid (see, e.g., \cite{GrayAl2018,Fontan2018Multiequilibria,zhang2025online}).
\begin{assumption}[Nonlinearity]\label{assumption:psi}
Each nonlinear function $\psi_i(x_i):\R\to \R$ of the vector $\psi(x)$ satisfies the following conditions:
\begin{gather}
\psi_i(x_i)=-\psi_i(-x_i),\,\forall x_i\in \R\;\;\text{(odd)}
\tag{A.1}
\label{assumption:1psiOdd}
\\
\pde{\psi_i}{x_i}(x_i)>0\;\forall x_i\in \R\;\text{and }\pde{\psi_i}{x_i}(0)=1\;\;\text{(monotone)}
\tag{A.2}
\label{assumption:2psiMonotone}
\\
\lim_{x_i\to\pm \infty} \psi_i(x_i)=\pm 1\;\;\text{(saturated)}
\tag{A.3}
\label{assumption:3psiSaturated}
\\
\psi_i(x_i) \; 
\begin{cases}
\text{strictly convex} & \forall\, x_i<0\\
\text{strictly concave}& \forall\, x_i>0
\end{cases} \;\;\text{(sigmoidal)}
\tag{A.4}
\label{assumption:4Sigmoidal}
\\
\psi_i(\varepsilon)=\psi_j(\varepsilon)=:\psi_u(\varepsilon),\;\forall i,j\in\{1,\dots,n\},\; \varepsilon\in \R\notag\\\qquad\qquad\qquad \text{(identical nonlinearities)}.
\tag{A.5}
\label{assumption:5psiIdentical}
\end{gather}
\end{assumption}
We now introduce the main assumption concerning the structure of the hypergraph.
\begin{assumption}[Hypergraph]\label{assumption:Ai}
Each matrix $\Bi{i}$, $i=1,\dots,n$, satisfies the following conditions 
\begin{enumerate}[label=(\roman*)]
    \item Undirected edges: $\Bi{i} = \Bi{i}^T$;
    \item No self-loops: $[\Bi{i}]_{ij} = [\Bi{i}]_{ji} = 0$, $\forall j\in\mathcal V$; 
    
    \item Proportional influence: $\exists\, \alpha$ such that $\sum_{j,k=1}^n [\Bi{i}]_{jk} = \alpha \sum_{j=1}^n a^{(2)}_{ij}$, or, equivalently, $\1_n^T \Bi{i} \1_n = \alpha [\A\1_n]_i$.
\end{enumerate}
\end{assumption}
Assumption~\ref{assumption:Ai}(i)--(ii) state that each $\Bi{i}$ is symmetric and that each 2-interaction can only involve three distinct agents; intuitively, Assumption~\ref{assumption:Ai}(ii) extends the concept of a null diagonal (meaning no self-loops) from 1-interactions to 2-interactions.
In addition, Assumption~\ref{assumption:Ai}(iii) describes a proportional influence scenario, where all agents share the same proportion of engagements between 1-interactions and 2-interactions. This assumption, together with Assumption~\ref{assumption:psi} \eqref{assumption:5psiIdentical}, is essential for analyzing consensus equilibria in Section~\ref{sec:main-results}.

\section{Existence and stability of multiple equilibria}\label{sec:main-results}
Before presenting the main results, we recall from \cite{GrayAl2018,Fontan2018Multiequilibria,Fontan2021Role} the collective behavior of system~\eqref{eqn:model} when $h=1$. If $\pi \in (0,1)$, the origin is the unique equilibrium point of system~\eqref{eqn:model_h1} and is globally asymptotically stable. When $\pi=\pi_1=1$, the system~\eqref{eqn:model_h1} undergoes a pitchfork bifurcation: The origin becomes unstable, and two new alternative consensus equilibria (i.e., $x^\ast \in \vspan{\1_n}$) emerge, both of which are locally asymptotically stable for any $\pi\in (\pi_1,\pi_2)$, where $\pi_2=\lambda_{n-1}(\Delta^{-1}A)$, where $A=[a_{ij}]$ in \eqref{eqn:model_h1}. Finally, when $\pi = \pi_2$, the system undergoes a second pitchfork bifurcation, and for $\pi > \pi_2$ new equilibria appear.

In this section, we aim to investigate the collective behavior of the system~\eqref{eqn:model} for $h=2$ as $\pi>0$ increases and to examine the conditions under which the system exhibits consensus equilibria.

\subsection{The equilibrium at the origin}
First, observe that the origin is an equilibrium point of the system~\eqref{eqn:model} for all $\pi>0$. Moreover, there exists a threshold value $\pi_1\ge 1$ such that, for any $\pi \in (0,\pi_1)$, it is locally asymptotically stable, which can be shown by using Lyapunov's indirect method (Lemma~\ref{lemma:origin}). 
However, unlike the 1-interactions case of model~\eqref{eqn:model_h1}, the presence of higher-order terms makes the analysis of the global stability of the origin more complex, particularly in the derivation of a value $\pi_1^\ast \in (0,\pi_1)$ such that the origin is globally asymptotically stable for \eqref{eqn:model} for any $\pi<\pi_1^\ast$. While we are only able to derive a conservative bound $\tilde{\pi}_1 \le \pi_1^\ast$ in the general case (Theorem~\ref{thm:origin_GAS_conservative}), under the assumption of identical nonlinear functions $\psi(\cdot)$ and a proportional influence assumption for the adjacency tensor $\Bd$ (see Assumption~\ref{assumption:Ai}(iii)), it is possible to obtain an explicit expression for $\pi_1^\ast$
(Theorem~\ref{thm:origin_GAS}).

We begin by establishing the existence of the trivial equilibrium and analyzing its local stability in the following lemma. Denote 
\begin{equation}
    \pi_1:= \frac{1}{\lambda_n(\Delta^{-1}\A)} \ge 1. 
    \label{eqn:pi1}
\end{equation}

\begin{lemma}\label{lemma:origin}
Suppose that \eqref{assumption:1psiOdd}--\eqref{assumption:4Sigmoidal} in Assumption~\ref{assumption:psi} and Assumption~\ref{assumption:Ai} (i)--(ii) hold for all $i\in \mathcal V$ in system~\eqref{eqn:model}, then the origin is an equilibrium point of system \eqref{eqn:model} for all $\pi>0$. Moreover, the origin is locally asymptotically stable if $\pi \in (0,\pi_1)$, and unstable if $\pi>\pi_1$.
\end{lemma}
\begin{proof}
The first claim follows directly from \eqref{assumption:1psiOdd}.

The Jacobian of the system~\eqref{eqn:model} at the origin is given by $J\!=\!-\Delta + \pi~\A~\pde{\psi}{x}(0)\!=\!-\Delta+\pi \A\!=\!\Delta(-I\!+\pi~\Delta^{-1}\A)$.
Observe that $\delta_i = \sum_{j} a^{(2)}_{ij} + \1_n^T \Bi{i}\1_n$. 
Since $\Delta$ is a diagonal matrix with strictly positive diagonal elements and $J$ is symmetric, the eigenvalues of $J$ are strictly negative for any $\pi \in (0,\pi_1)$, ensuring local asymptotic stability of the origin. However, if $\pi > \pi_1$, the Jacobian $J$ has at least one positive eigenvalue, which means that $x=0$ is unstable. Finally, using the Ger\u{s}gorin's Theorem \cite[Thm 6.1.1]{HornJohnson2013} and recalling that $\Delta\ge 0$ and $\A \ge 0$, we obtain that every eigenvalue $\lambda_i(\Delta^{-1}\A)$ of $\Delta^{-1}\A$, $i=1\dots,n$, satisfies
\begin{equation*}
|\lambda_i(\Delta^{-1}\A)| \le \sum_{j=1}^n \frac{1}{\delta_i} a_{ij}^{(2)} \le 1
\end{equation*}
As $\lambda_n(\Delta^{-1}\A)>0$ then $\lambda_n(\Delta^{-1}\A)\le 1$ and $\pi_1\ge 1$.
\end{proof}

% ============================ CONSERVATIVE
The following result provides a sufficient condition under which the origin is globally asymptotically stable.

\begin{theorem}\label{thm:origin_GAS_conservative}
Suppose that \eqref{assumption:1psiOdd}--\eqref{assumption:4Sigmoidal} in Assumption~\ref{assumption:psi} and Assumption~\ref{assumption:Ai} (i)--(ii) hold for all $i\in \mathcal V$ in system~\eqref{eqn:model}.
The origin is a globally asymptotically stable equilibrium point of system \eqref{eqn:model} for all $\pi\in (0,\tilde{\pi}_1)$, where $\tilde{\pi}_1=\frac{1}{\lambda_n((H+H^T)/2)} \le 1$ and the matrix $H$ is given by 
\begin{equation}
H:= \Delta^{-1} \Bigg(\A + \begin{bmatrix}
\1^T \Bi{1} \\ \vdots \\\1^T \Bi{n}
\end{bmatrix}\Bigg).\label{eqn:H}
\end{equation} 
\end{theorem}
\begin{proof}
As in \cite{Fontan2021Role}, let $V : \R^n \to \R_+$ be the Lyapunov function described by 
\begin{equation}
V(x)=\sum_{i=1}^{n}\,\frac{1}{\delta_i}\int_{0}^{x_i} \psi_i(\varepsilon)\,d\varepsilon.
\label{eqn:Lyapunov}
\end{equation}
Since each function $\psi_i(\cdot)$ is monotonically increasing and $\psi_i(\varepsilon)=0$ if and only if $\varepsilon=0$, then $V(x)>0$ for all $x\in  \R^n\setminus \{0\}$ and $V(0)=0$. Moreover, $V(x)$ is radially unbounded.

Computing the derivative of $V$ along the trajectories gives
\begin{align*}
&\dot V(x) = \psi(x)^T \Delta^{-1}\dot x
\\
& = \psi(x)^T \Bigg(\!\!\!
- x + \pi \Delta^{-1} \Big(\A \psi(x) +\! \begin{bmatrix}
\psi(x)^T \Bi{1} \psi(x) \\ \vdots \\\psi(x)^T \Bi{n}\psi(x)
\end{bmatrix} \Big)\Bigg)
\\
& = -\psi(x)^T x 
+ \pi \psi(x)^T\Delta^{-1} \Bigg(\A +\! \begin{bmatrix}
\psi(x)^T \Bi{1} \\ \vdots \\\psi(x)^T \Bi{n}
\end{bmatrix}\Bigg)\psi(x)
\\
& \le -\psi(x)^T x 
+ \pi |\psi(x)^T|\Delta^{-1} \Bigg(\A +\! \begin{bmatrix}
\1_n^T \Bi{1} \\ \vdots \\\1_n^T \Bi{n}
\end{bmatrix}\Bigg)|\psi(x)|
\\
& = -\psi(x)^T x + \pi |\psi(x)^T| H |\psi(x)|
\\
& < -|\psi(x)^T| |\psi(x)| + \pi |\psi(x)^T| H |\psi(x)|
\\
& = |\psi(x)^T| (-I + \pi H ) |\psi(x)|
\\
& = |\psi(x)^T| \Big(-I + \pi \,\frac{H+H^T}{2} \Big) |\psi(x)|,
\end{align*}
Therefore, if $\pi < 1/\lambda_n((H+H^T)/2)$ then the matrix $-I+\pi \frac{H+H^T}{2}$ is negative definite and $\dot{V}(x)<0$, which proves that the origin is globally asymptotically stable. 

Finally, observe that $H$ is nonnegative and irreducible, hence the Perron-Frobenius Theorem \cite[Thm 1.4]{HornJohnson2013} holds. Therefore, from the definition of $\Delta$, $H\1_n = \1_n$, which means that $\lambda_n(H)=1$; then, $\lambda_n((H+H^T)/2 \ge \lambda_n(H) = 1$ and $1/\lambda_n((H+H^T)/2)\le 1$.
\end{proof}

Next, we show that by additionally assuming identical nonlinearities and proportional influence in system~\eqref{eqn:model}, the condition on $\pi$ required for the origin to be globally asymptotically stable can be relaxed. To this end, we define the nonlinear function $g: \R_+ \times \R_+ \to \R$ given by
\begin{equation}
    g(\varepsilon,\pi) := -(1+\alpha) \varepsilon + \pi(\psi_u(\varepsilon)+\alpha \psi_u(\varepsilon)^2),
\label{eqn:g}
\end{equation}
where $\psi_u$ is the identical nonlinear function in 
Assumption \ref{assumption:psi} \eqref{assumption:5psiIdentical}, and $\alpha$ is defined in Assumption \ref{assumption:Ai} (iii).

\begin{theorem}\label{thm:origin_GAS}
Suppose system~\eqref{eqn:model} satisfies all conditions in  Assumption \ref{assumption:psi} and Assumption \ref{assumption:Ai}, then the origin is a globally asymptotically stable equilibrium point of system~\eqref{eqn:model} for any $\pi<\pi_1^\ast<\pi_1$, where $\pi_1$ is given by \eqref{eqn:pi1} and $\pi_1^\ast$ solves
\begin{equation}
    g(\varepsilon,\pi^\ast_1)=0,\quad \pde{g}{\varepsilon}(\varepsilon,\pi_1^\ast)=0,\quad
    \varepsilon > 0.
\label{eqn:pi1_ast_exact}
\end{equation}
\end{theorem}

\begin{proof}
Let $V(x(t))=\norm{x(t)}_{\infty} = \max_i |x_i(t)|$ and observe that $V(x)\ge 0$ for all $x\in \R^n$ and $V(x)=0$ if and only if $x=0$.
Following \cite{Blanchini1995nonquadratic}, let $Q = [I_n\; -I_n]$, $Q_i$ the $i$th column of $Q$, and $\mathrm{I}(x(t))$ the set of the indexes $i$ such that $V(x(t))=Q_i^T x(t)$. Observe that $Q_i^T x = x_i$ if $x_i>0$ and $Q_i^T x = -x_i$ if $x_i<0$. 
Then, the upper Dini derivative of $V$ along the trajectories \eqref{eqn:model} gives
\begin{align*}
d^+ V&(x)
=  \max_{i \in \mathrm{I}(x)} Q_i^T \dot x
\\
& =  \max_{i \in \mathrm{I}(x)} Q_i^T 
\Big( -\Delta x + \pi \A \psi(x)
\\&\qquad\qquad+ \pi (I_n \otimes \psi(x)^T) \Bd (\1_n \otimes \psi(x)) \Big)
\\
& =  \max_{i \in \mathrm{I}(x)}  
\Big( -\delta_i |x_i| + \pi \,\sign{x_i} \sum_{j=1}^n a^{(2)}_{ij}  \psi_j(x_j)
\\&\qquad\qquad+ \pi \,\sign{x_i} \sum_{j,k=1}^n a^{(3)}_{ijk}  \psi_j(x_j)\psi_k(x_k) \Big)
\end{align*}
Let $i\in \mathrm{I}(x)$. Observe that $\sign{x_i} \psi_j( x_j) \le |\psi_j( x_j)|\le |\psi_i(x_i)|$ for \eqref{assumption:2psiMonotone}. Similarly, 
$\sign{x_i}\psi_j(x_j)\psi_k(x_k) \le |\psi_i(x_i)|^2$. Then, recalling that $\sum_{j,k=1}^n a^{(3)}_{ijk}=\alpha \sum_{j=1}^n a^{(2)}_{ij}$ and $\delta_i = (1+\alpha) \sum_{j=1}^n a^{(2)}_{ij}$, we obtain:
\begin{align*}
d^+ V(x)
&\le  \max_{i \in \mathrm{I}(x)}  
\Big( \sum_{j=1}^n a^{(2)}_{ij} g(|x_i|,\pi) \Big)
\end{align*}
where the function $g$ is defined in \eqref{eqn:g}.

The scalar function $\tilde{g}(\varepsilon,\pi)= \pi \big(\psi(\varepsilon) +\alpha \psi(\varepsilon)^2\big)$ is positive for any $\varepsilon>0$, is equal to $0$ at $\varepsilon=0$, and is saturated. Moreover, the following holds:
\begin{subequations}
\begin{align}
\lim_{\varepsilon\to +\infty} \tilde{g}(\varepsilon,\pi) &= (1+\alpha) \pi
\\
\pde{\tilde g}{\varepsilon}(0,\pi)&= \pi > 0
\\
\ppde{\tilde g}{\varepsilon}(0,\pi)&= 
2\alpha \pi > 0 
\end{align}
\end{subequations}
This implies that, when $\pi = 1+\alpha$, there exists $\varepsilon^\ast$ s.t. $\tilde{g}(\varepsilon^\ast,1+\alpha)=(1+\alpha) \varepsilon^\ast$. In turn, this means that there exist a $\pi^\ast\in (0,1+\alpha)$ s.t. $\tilde{g}(\varepsilon,\pi^\ast)$ is tangent to $(1+\alpha)\varepsilon$.
Therefore, for any value of $\pi$ below $\pi^\ast$, where $\pi^\ast$ solves \eqref{eqn:pi1_ast_exact}, $g(\varepsilon,\pi)<0$ for any $\varepsilon >0$. Finally, this implies that $d^+ V(x) <0$ for any $x \in \R^{n}\setminus\{0\}$, which concludes the proof.
\end{proof}

To conclude this subsection, we note that the thresholds in Lemma~\ref{lemma:origin}, Theorem~\ref{thm:origin_GAS_conservative}, and Theorem~\ref{thm:origin_GAS} satisfy the relation $\tilde{\pi}_1 \leq \pi_1^\ast \leq \pi_1$, because
\begin{align}\label{eqn:all-the-pi1}
    \tilde{\pi}_1 \le 1 \le \pi_1^\ast \le \pi_1,
\end{align}
where $\pi_1^\ast$ solves \eqref{eqn:pi1_ast_exact}, $\pi_1\!=\!\frac{1}{\lambda_n(\Delta^{-1}\A)}$, and  $\tilde{\pi}_1\!=\!\frac{1}{\lambda_n(\frac{H+H^T}{2})}$ with $H = \Delta^{-1}(\A + [\Bi{1}^T\1_n \,\cdots \Bi{n}^T\1_n ]^T)$. Note that $\lambda_n(H)=1$.

\subsection{Existence of nontrivial equilibria}
The previous section primarily focuses on the trivial equilibrium point, the origin. 
In contrast, this section investigates the existence of nontrivial equilibria and presents a conjecture regarding their stability (Lemma~\ref{lemma:nontrivial-eq-stability}).

We derive a necessary condition on the parameter $\pi$ for the existence of nontrivial equilibrium points, namely $\pi\ge \frac{1}{\lambda_n(H)}$, where the matrix $H$ was introduced in \eqref{eqn:H} (Theorem~\ref{thm:nontrivial-eq-necessity}). From \eqref{eqn:all-the-pi1}, note that there is a gap between this threshold and the (conservative) one proposed in Theorem~\ref{thm:origin_GAS_conservative}, i.e., $\tilde{\pi}_1=\frac{1}{\lambda_n((H+H^T)/2)}$.
Then, using singularity and unfolding theory \cite{GolubitskySchaeffer1985}, we show that the presence of higher-order interactions (compared to the case $h=1$) leads, at $\pi=\pi_1$, to an unfolding of the pitchfork bifurcation that was observed for the model~\eqref{eqn:model_h1} (Theorem~\ref{thm:nontrivial-eq-sufficiency}). Since the methods used to derive these results are local, we implicitly assume that the parameters describing the higher-order interactions, $\Bi{i}$ $i=1,\dots,n$, are small.
Finally, under the same simplifying assumptions considered in Theorem~\ref{thm:origin_GAS}, namely, identical nonlinearities $\psi(\cdot)$ and Assumption~\ref{assumption:Ai}(iii) for the 2-interactions adjacency tensor $\Bd$, it is possible to show that the system~\eqref{eqn:model} admits two equilibrium points that are of consensus for any $\pi_1 > \pi_1^\ast$, where $ \pi_1^\ast$ solves \eqref{eqn:pi1_ast_exact} (Lemma~\ref{lemma:nontrivial-eq-consensus}).

To analyze the nontrivial equilibria of system~\eqref{eqn:model}, we begin by establishing a necessary condition for their existence.

% ================ necessary condition
\begin{theorem}\label{thm:nontrivial-eq-necessity}
Suppose that \eqref{assumption:1psiOdd}--\eqref{assumption:4Sigmoidal} in Assumption~\ref{assumption:psi} and Assumption~\ref{assumption:Ai} (i)--(ii) hold for all $i\in \mathcal V$ in system~\eqref{eqn:model}.
If $\xeq\in\R^n\setminus\{0\}$ is an equilibrium point of system~\eqref{eqn:model}, then 
\[
\pi\ge \frac{1}{\lambda_n(H)}=1.
\]
\end{theorem}
\begin{proof}
Let 
\begin{equation*}
f(x):= \Delta^{-1} \Bigg(\A\psi(x)+ \begin{bmatrix} \psi(x)^T \Bi{1}\psi(x)\\ \vdots\\ \psi(x)^T \Bi{n}  \psi(x) 
\end{bmatrix}\Bigg),
\end{equation*}
and note that $\xeq$ is an equilibrium point of \eqref{eqn:model} if and only if $\pi f(\xeq)= \xeq$. Since $|\psi(x)|\le \1_n$ for all $x\in \R^n$, $\A\ge 0$, and $\Bi{i}\ge 0$ for all $i=1,\dots,n$, $f(\xeq)$ can be upper bounded by
\begin{align*}
f(\xeq) 
& \le \Delta^{-1} \Bigg(\A +\! \begin{bmatrix}
\1_n^T \Bi{1} \\ \vdots \\\1_n^T \Bi{n}
\end{bmatrix}\Bigg)|\psi(\xeq)|
 = H |\psi(\xeq)|
\end{align*}
where the matrix $H$ is defined in \eqref{eqn:H}.

Recall that $|\psi(\xeq)|\lneq |\xeq|$, i.e., $|\psi(\xeq)|\le |\xeq|$ and $|\psi_i(\xeq_i)|<|\xeq_i|$ for at least one $i$, where the latter condition follows from $\xeq\ne 0$. 
This, together with $\pi f(\xeq)=  \xeq$, gives
\begin{align*}
|\xeq| &\lneq \pi H |\xeq| ,
\end{align*}
which means \cite[Thm 8.3.2]{HornJohnson2013} that $ \pi \lambda_n(H)\ge 1$, i.e., $ \pi \ge \frac{1}{\lambda_n(H)}$.
\end{proof}

% ================ sufficient condition
Next, we apply singularity and unfolding theory (see, e.g., \cite{GolubitskySchaeffer1985} for details) to demonstrate that the inclusion of higher-order interactions in system~\eqref{eqn:model_h1} leads to an unfolding of the pitchfork bifurcation observed in the original model.

\begin{theorem}
\label{thm:nontrivial-eq-sufficiency}
Suppose that \eqref{assumption:1psiOdd}--\eqref{assumption:4Sigmoidal} in Assumption~\ref{assumption:psi} and Assumption~\ref{assumption:Ai} (i)--(ii) hold for all $i\in \mathcal V$ in system~\eqref{eqn:model}. Assume also that $\lambda_n(\Delta^{-1}\A)$ is simple.
Given $\pi_1$ as in \eqref{eqn:pi1}, let $g(y,\pi,\Bd)$ be the Lyapunov-Schmidt reduction of \eqref{eqn:model} at 
$(0,\pi_1,0)$. 
Then:
\begin{enumerate}[label=(\roman*)]
    \item The bifurcation problem $g(y,\pi,0)$ has a symmetric pitchfork singularity at $(0,\pi_1)$: When $\pi=\pi_1$, the system~\eqref{eqn:model} undergoes a pitchfork bifurcation, the origin becomes unstable, and two new equilibria appear. These equilibria are locally asymptotically stable;
    \label{thm:nontrivial-eq-sufficiency-i}

    \item For $\Bd\ne 0$, the bifurcation problem $g(y,\pi,\Bd)$ is an $n^3$-parameter unfolding of the symmetric pitchfork.
    \label{thm:nontrivial-eq-sufficiency-ii}
\end{enumerate}
\end{theorem}
\begin{proof}
The proof uses singularity and unfolding theory \cite{GolubitskySchaeffer1985} and extends \cite{Fontan2018Multiequilibria} to the higher-order interactions case. 

Item \ref{thm:nontrivial-eq-sufficiency-i} follows directly from \cite{Fontan2018Multiequilibria} or \cite[Thm 1]{Fontan2021Role}, where it is proven that the Lyapunov-Schmidt reduction of \eqref{eqn:model} at $(0,\pi_1,0)$ satisfies the following conditions at $(0,\pi_1,0)$:
\begin{equation*}
g = g_y = g_\pi = g_{yy} =0,\quad g_{yyy} < 0,\quad g_{\pi y} > 0.
\end{equation*}
Since these conditions solve the recognition problem for a pitchfork bifurcation, \ref{thm:nontrivial-eq-sufficiency-i} is then proven. 

\vspace{0.1cm}

Item \ref{thm:nontrivial-eq-sufficiency-ii} is the definition of an unfolding. In particular, using the derivations from \cite{Fontan2018Multiequilibria} and \cite[eq. \S I.3.23(d)]{GolubitskySchaeffer1985}, yielding $g_{[\Bi{i}]_{jk}}(0,\,\pi_1,0) = 0$ and 
$ g_{yy[\Bi{i}]_{jk}}(0,\pi_1,0) = 2 \pi_1 \delta_i^{-1} w_i v_j v_k >0$ ($ v>0$ and $w>0$ are the right and left eigenvectors relative to $ \lambda_{n} (\Delta^{-1}\A )$) for all $i,k,j=1,\dots,n$, then the Lyapunov-Schmidt reduction of \eqref{eqn:model} at $(0,\pi_1,0)$ is equivalent to
\begin{equation}
    \dot y = (\pi-\pi_1) y +  \kappa_1 y^3 + \kappa_2 y^2, \quad \kappa_1<0,\kappa_2>0.
\label{eqn:normal-form-unfolding}
\end{equation}
\end{proof}
From Theorem~\ref{thm:nontrivial-eq-sufficiency} and the condition $\kappa_2>0$ in \eqref{eqn:normal-form-unfolding}, it follows that the unfolding of the pitchfork bifurcation  exhibits a upper branch of equilibria.

% ================ what happens under the simplifying assumptions
\begin{lemma}\label{lemma:nontrivial-eq-consensus}
Suppose system~\eqref{eqn:model} satisfies all conditions in  Assumption \ref{assumption:psi} and Assumption \ref{assumption:Ai}.
Let $g: \R_+ \times \R_+ \to \R$ be the nonlinear function defined in \eqref{eqn:g}.
Then, for each $\pi>\pi_1^\ast$, where $\pi_1^\ast$ solves \eqref{eqn:pi1_ast_exact}, the system~\eqref{eqn:model} admits two consensus equilibrium points $\varepsilon_{1} \1_n, \varepsilon_2 \1_n$, where $\varepsilon_1,\varepsilon_2$ are the two solutions of $g(\varepsilon,\pi)=0$.
\end{lemma}
\begin{proof}
Let 
\begin{equation*}
f(x):= \A\psi(x)+ \begin{bmatrix} \psi(x)^T \Bi{1}\psi(x)\\ \vdots\\ \psi(x)^T \Bi{n}  \psi(x) 
\end{bmatrix},
\end{equation*}
and note that $x=\varepsilon \1_n$ is an equilibrium point of \eqref{eqn:model} if and only if $\pi f(\varepsilon \1_n)=\varepsilon 
 \Delta \1_n$. Under \eqref{assumption:5psiIdentical}, $\psi(\varepsilon \1_n) = \psi_u(\varepsilon) \1_n$, and $f(\varepsilon \1_n)$ is given by
\begin{align*}
f(\varepsilon\1) 
&=  \psi_u(\varepsilon)
\A \1_n + \psi_u(\varepsilon)^2 \begin{bmatrix} \1_n^T \Bi{1}\1_n\\ \vdots\\ \1_n^T \Bi{n}  \1_n 
\end{bmatrix}
\\
&=\psi_u(\varepsilon)
\A \1_n + \psi_u(\varepsilon)^2 \alpha \A \1_n 
\\
&=\A \1_n \big(\psi_u(\varepsilon)  + \alpha \psi_u(\varepsilon)^2 \big)
\end{align*}
where the second equality holds due to Assumption~\ref{assumption:Ai}(iii).

Assumption~\ref{assumption:Ai}(iii) also implies that 
$\delta_i = \sum_{j=1}^n a^{(2)}_{ij}  + \1_n \Bi{i} \1_n  
= \sum_{j=1}^n a^{(2)}_{ij} (1+\alpha)$ for all $i=1,\dots,n$, i.e, $\Delta = (1+\alpha)\, \diag{\A\1_n}$.
Therefore, $\varepsilon \Delta \1_n = \varepsilon (1+\alpha)\, \A\1_n$.

It follows that $\pi f(\varepsilon \1_n)=\varepsilon \Delta \1_n$ holds if and only if 
$\pi \big(\psi_u(\varepsilon)  + \alpha \psi_u(\varepsilon)^2 \big) =(1+\alpha) \varepsilon $, or, equivalently, if and only if $g(\varepsilon,\pi)=0$, with $g$ as in \eqref{eqn:g}. The existence of solutions for $g(\varepsilon,\pi)=0$ for $\pi>\pi_1^\ast$ follows the proof of Theorem~\ref{thm:origin_GAS}.
\end{proof}

\subsection{Stability of nontrivial equilibria}
Regarding the local asymptotic stability of the new nontrivial equilibrium points, it can be shown, using similar arguments to \cite[Theorem 2(ii.2)]{Fontan2021Role}, that for any $\pi\in (\pi_1,\pi_2)$ with $\pi_2 = \frac{1}{\lambda_{n-1}(\Delta^{-1}A)}$, every nontrivial equilibrium point is locally asymptotically stable (Lemma~\ref{lemma:nontrivial-eq-stability}). 
Finally, using similar arguments to \cite[Theorem 3(ii)]{Fontan2021Role}, we can prove that all the solutions of \eqref{eqn:model} are bounded (Lemma~\ref{lemma:nontrivial-eq-set-stability}). 
Due to the length of the proofs and their similarity to the ones from \cite{Fontan2021Role}, we choose not to include them in this paper and will instead present them in future work.
\begin{lemma}\label{lemma:nontrivial-eq-stability}
Suppose that \eqref{assumption:1psiOdd}--\eqref{assumption:4Sigmoidal} in Assumption~\ref{assumption:psi} and Assumption~\ref{assumption:Ai} (i)--(ii) hold for all $i\in \mathcal V$ in system~\eqref{eqn:model}.
Let $\xeq\in \R^n\setminus\{0\}$ be an equilibrium point of \eqref{eqn:model} at $\pi\in (\pi_1,\pi_2)$, where $\pi_2 = \frac{1}{\lambda_{n-1}(\Delta^{-1}\A)}$. Then, $\xeq$ is locally asymptotically stable.
\end{lemma}

\begin{lemma}\label{lemma:nontrivial-eq-set-stability}
Suppose that \eqref{assumption:1psiOdd}--\eqref{assumption:4Sigmoidal} in Assumption~\ref{assumption:psi} and Assumption~\ref{assumption:Ai} (i)--(ii) hold for all $i\in \mathcal V$ in system~\eqref{eqn:model}.
Then, the trajectories of the system~\eqref{eqn:model} asymptotically converge to the set $\Omega = \{x\in \R^n: \norm{x}_\infty\le \pi\}$.
\end{lemma}

\section{Numerical example}\label{sec:example}
Consider an hypernetwork $\G^{(2)}$ of $n = 5$ nodes and weighted adjacency tensors $\A$ and $\Bi{i}$ $i=1,\dots,n$. 
The matrix $\A$ is nonnegative, symmetric, and irreducible; its elements are drawn (single realization) from a uniform distribution (with $P[a^{(2)}_{ij}\ne 0] = 0.8$ for all $i,j$).
Similarly, each matrix $\Bi{i}$ is nonnegative and symmetric, and its elements are drawn from a uniform distribution with $P[[\Bi{i}]_{jk}\ne 0] = 0.2$ for all $j,k$. Moreover, they are scaled to satisfy Assumption~\ref{assumption:Ai}(iii) with $\alpha =1$. Figure~\ref{fig:EX1_graphs} provides an illustration of all the adjacency tensors.

We obtain $\pi_1 = 1+\alpha = 2$, $\pi_1^\ast \approx 1.44$, and $\tilde{\pi}_1\approxeq 0.9$. Figure~\ref{fig:EX1_bifurcation-diagram} depicts the bifurcation diagram of the system~\eqref{eqn:model} for a component $x_i$, showing the equilibrium points as we vary $\pi=0.005,0.01,\dots,5$. 

We first consider the case where there are no 2-interactions among the agents, i.e., $\Bd=0$, which simplifies the system~\eqref{eqn:model} to \eqref{eqn:model_h1}. The 1-interactions model \eqref{eqn:model_h1} exhibits a symmetric pitchfork bifurcation at $\pi_1$: For any $\pi<\pi_1$, the origin is the unique equilibrium point and is globally asymptotically stable \cite{Fontan2018Multiequilibria}; For any $\pi>\pi_1$ and $\pi$ close to $\pi_1$, system \eqref{eqn:model_h1} admits two alternative equilibrium points $\pm \xeq\in \vspan{\1_n}$, which are locally asymptotically stable (Figure~\ref{fig:EX1_bifurcation-diagram-a}). 

When 2-interactions are taken into account and Assumption~\ref{assumption:Ai}(iii) holds, we obtain the bifurcation diagram of Figure~\ref{fig:EX1_bifurcation-diagram-b}, which illustrates the unfolding of a pitchfork bifurcation. For any $\pi<\pi_1^\ast$, where $\pi_1^\ast$ solves \eqref{eqn:pi1_ast_exact}, the origin is the unique equilibrium point and it is globally asymptotically stable (Theorem~\ref{thm:origin_GAS}). Moreover, we know that it is locally asymptotically stable for any $\pi<\pi_1$ and becomes unstable at $\pi_1$ (Lemma~\ref{lemma:origin}). For any $\pi>\pi_1^\ast$ (and $\pi \not \gg \pi_1$), system \eqref{eqn:model_h1} admits two equilibrium points $\xeq_1, \xeq_2 \in \vspan{\1_n}$ with $\norm{\xeq_2}_2 < \norm{\xeq_1}_2$ (Theorem~\ref{thm:nontrivial-eq-sufficiency} and Lemma~\ref{lemma:nontrivial-eq-consensus}). $\xeq_2$ is unstable for any $\pi <\pi_1$.

This bifurcation diagram shows that for any $\pi\in (\pi_1^\ast,\pi_1)$, the system exhibits bistability: depending on the initial conditions, the trajectories of the system will either converge to the origin or (``jump'' and) converge to the stable nontrivial equilibrium point $\xeq_1$.

\subsection{An interpretation in the context of social networks}
This example is inspired by a previous longitudinal study done at the KTH Live-In Lab \cite{Fontan2023Social}, where the sustainability dynamics of $5$ tenants were observed over a period of $5$ weeks. The study revealed that, while all tenants actively participated with their neighbors in discussions related to sustainability within mobility, only certain subgroups—-such as triangles of tenants—-engaged in discussions on other topics like food, water consumption, and electricity consumption.
Therefore, similar to the co-authorship hypernetworks example \cite{Carletti2020random}, the social hypernetwork of the KTH Live-In Lab tenants captures (multi-)tenant collaborations across various sustainable activities (mobility, food, water consumption, and electricity consumption), which cannot be effectively represented by a traditional pairwise graph. In turn, this interpretation may lead to a deeper understanding of the decision-making process regarding sustainability efforts.

\begin{figure}[!t]\centering
\includegraphics[width=0.47\textwidth]{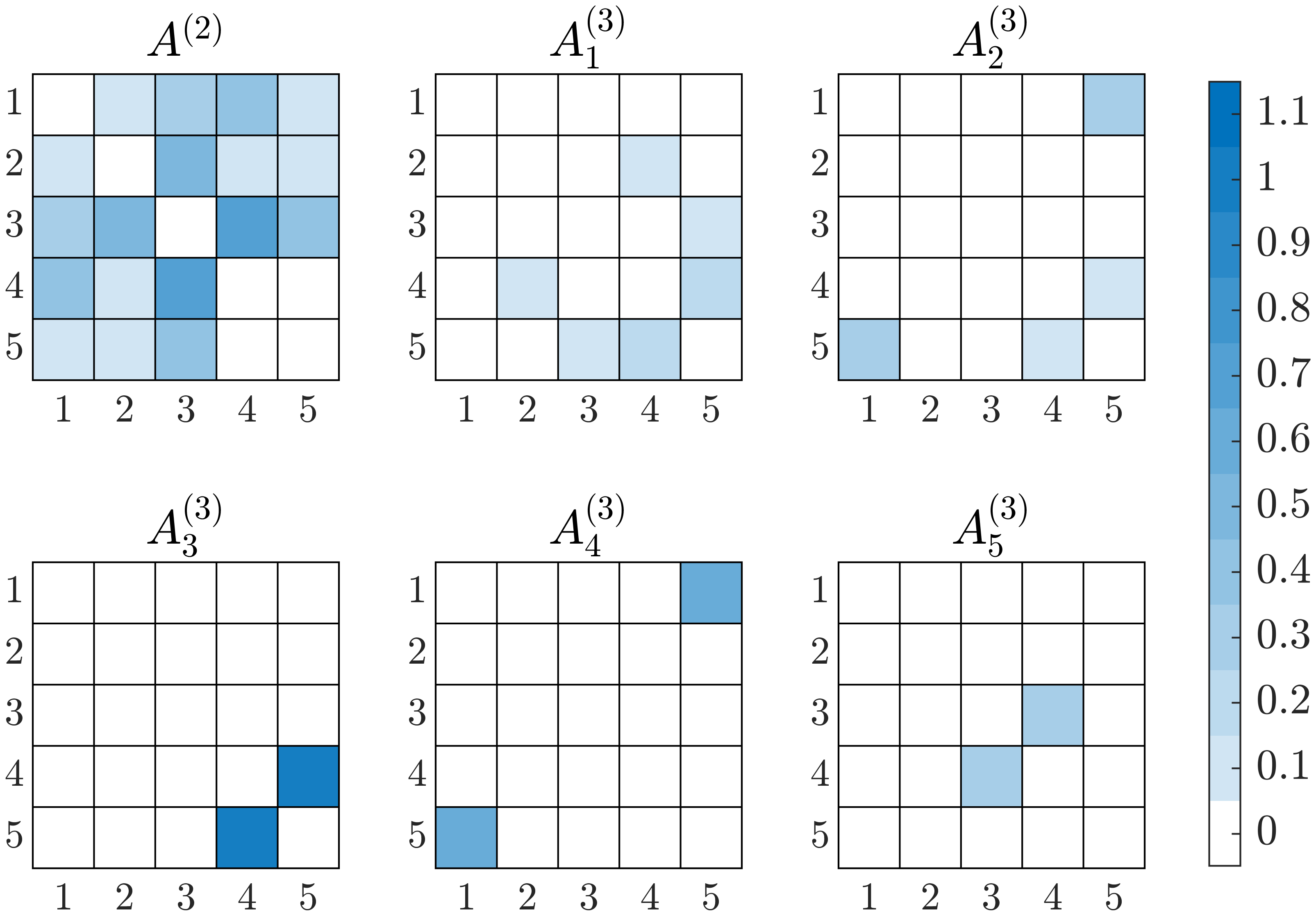}
    \caption{Hypergraph $\G^{2}$ and associated adjacency tensors $\A$ and $\Bi{i}$, $i=1,\dots,n$. The colormap represents the edge weights for both $1$- and 2-interactions.}
    \label{fig:EX1_graphs}
\end{figure}
\begin{figure}[!ht]\centering
\subfloat[]{\includegraphics[width=0.4\textwidth]{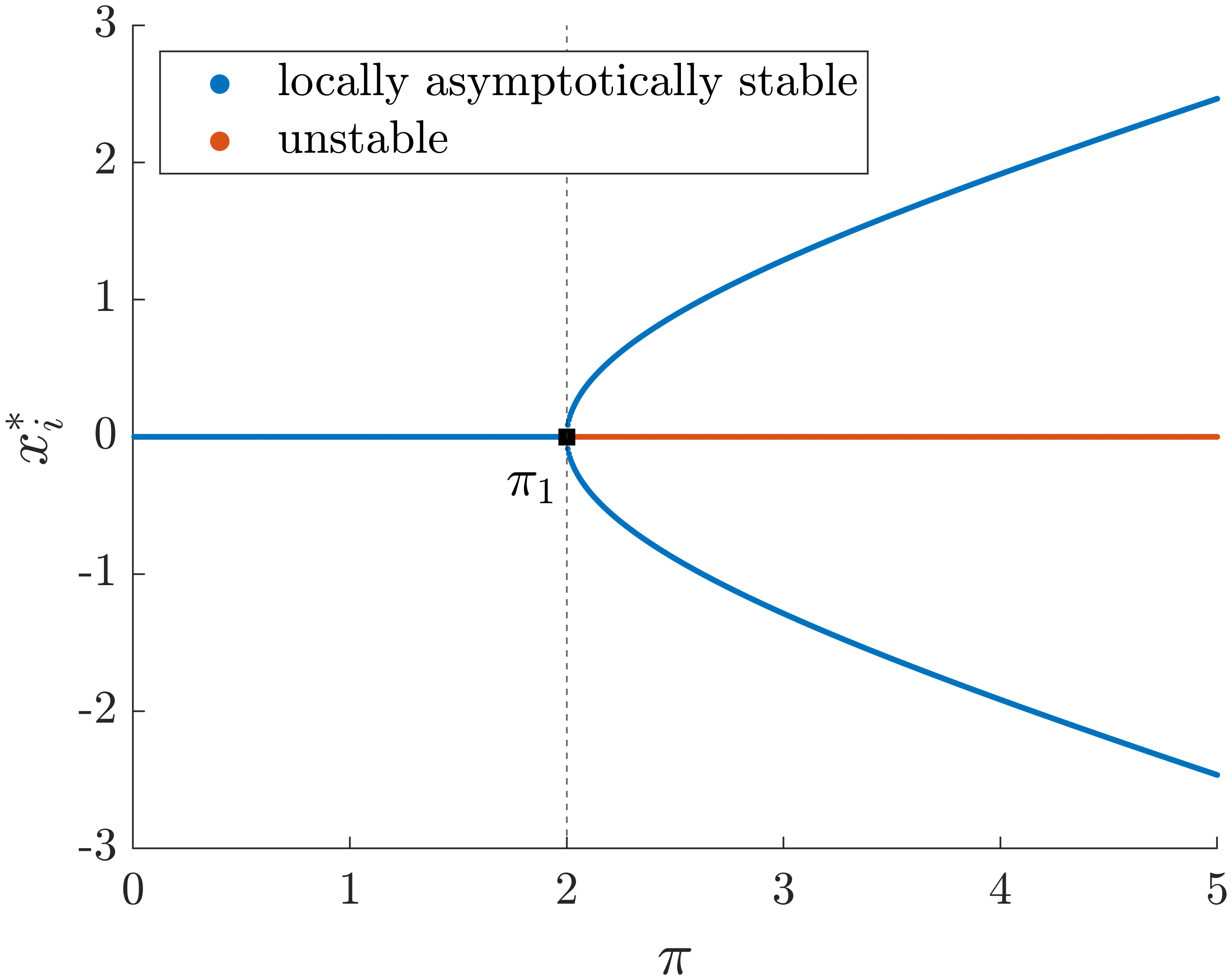}
\label{fig:EX1_bifurcation-diagram-a}}\\
\subfloat[]{\includegraphics[width=0.4\textwidth]{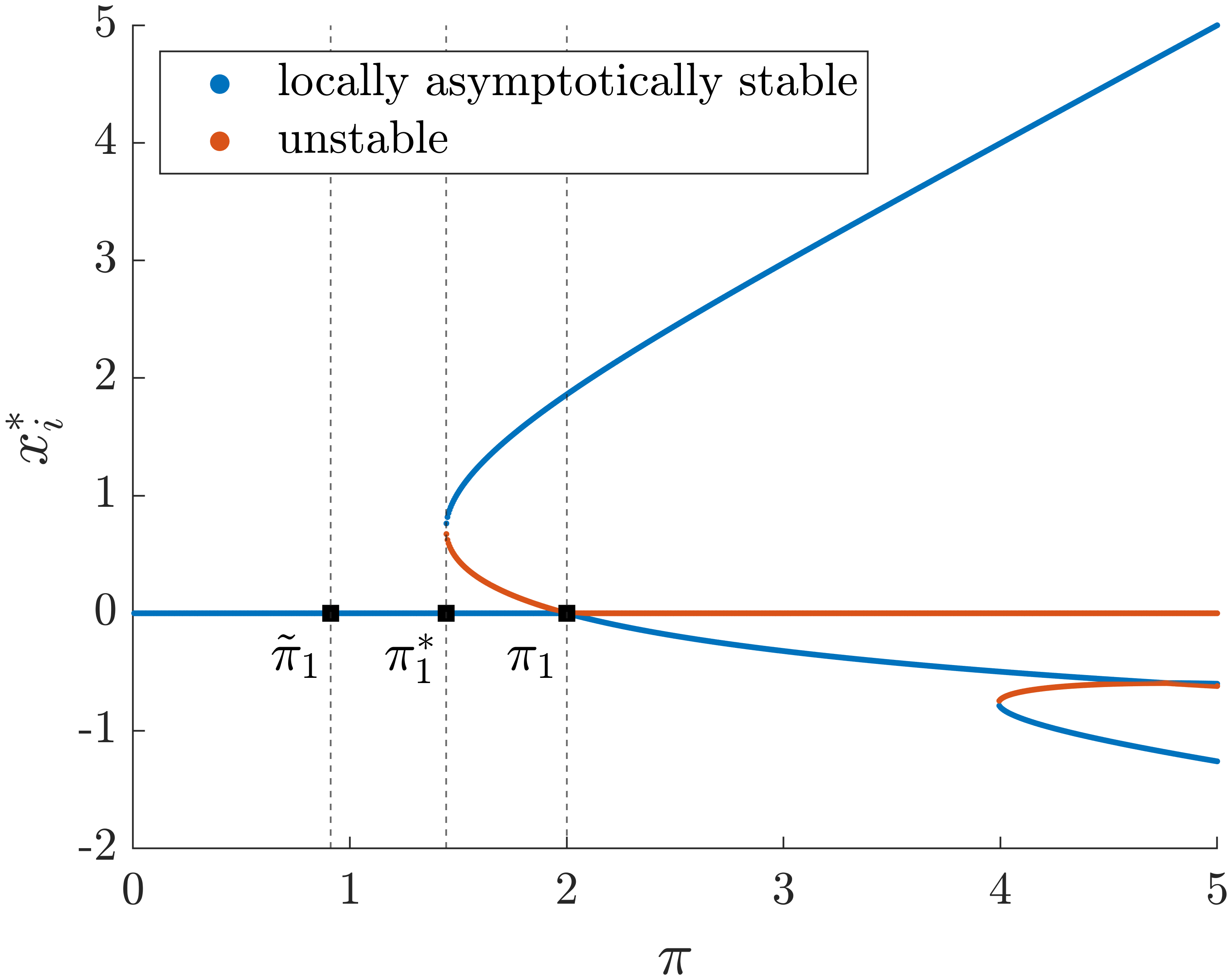}
\label{fig:EX1_bifurcation-diagram-b}}
\caption{Bifurcation diagram of system~\eqref{eqn:model}, depicted for a component $x_i$.
(a): Without 2-interactions, that is, $\Bd=0$. (b): With 2-interactions. The adjacency tensor $\Bd$ satisfies Assumption~\ref{assumption:Ai}(iii) with $\alpha=1$.}
\label{fig:EX1_bifurcation-diagram}
\end{figure}

\section{Conclusion}\label{sec:conclusion}
This study provides a qualitative analysis of a nonlinear dynamical system over cooperative hypernetworks, modeling a collective decision-making process in a multiagent system with higher-order interactions among the agents.
We show the system exhibits bistability through the unfolding of a pitchfork bifurcation, meaning that, depending on the initial conditions, the community of agents may either remain in a deadlock state or ``jump'' to a nontrivial decision. 

There are several potential directions for extending this work, the first of which is to complete the analysis in the case of $h$-interactions for $h\ge 2$. Other future directions include investigating the role of antagonistic interactions among agents and/or the impact of external influences on shaping the collective decision-making behavior.

\bibliographystyle{ieeetr}

%\bibliography{references,references-2}

\end{document}